\documentclass{article}
\usepackage{amsmath,amsthm,amssymb,graphicx}

\newtheorem{theorem}{Theorem}[section]
\newtheorem{lemma}[theorem]{Lemma}

\newtheorem{proposition}[theorem]{Proposition}
\newtheorem{definition}[theorem]{Definition}

\newtheorem{corollary}[theorem]{Corollary}

\theoremstyle{remark}

\begin{document}

\title{Graph colouring and the total Betti number}

\author{Alexander Engstr\"om\footnote{{\tt alexander.engstrom@aalto.fi}}}

\date\today

\maketitle

\begin{abstract}
The total Betti number of the independence complex of a graph is an intriguing graph invariant. Kalai and Meshulam have raised the question on its relation to cycles and the chromatic number of a graph, and a recent conjecture on that theme was proved by Bonamy, Charbit and Thomass\'e. We show an upper bound on the total Betti number in terms of the number of vertex disjoint cycles in a graph. The main technique is discrete Morse theory and building poset maps. Ramanujan graphs with arbitrary chromatic number and girth $\log n$ is a classical construction. We show that any subgraph of them with less than $n^{0.003}$ vertices have smaller total Betti number than some planar graph of the same order, although it is part of a graph with high chromatic number.
\end{abstract}

\section{Introduction}

Independence complexes of graphs provide an interesting link between graph theory and algebraic topology. Any (reasonable) CW-complex can be turned into an independence complex by further triangulations and studied with graph theory; conversely we can try to understand graph theoretic notions using algebraic topology.

The \emph{independence complex} $\textrm{Ind}(G)$ of a graph $G$ is the simplicial complex given by its independent sets and its total Betti number is
\[ {\mathbf{b}}(G) =  \sum_i \dim \tilde{H}_i(\textrm{Ind}(G)). \]
Many papers in both mathematics and physics have been written to understand ${\mathbf{b}}(G),$ see for example  \cite{Adam1,bin,E05,FSE05,hui,J06}.
Kalai and Meshulam \cite{Kpc}  have 
asked if the chromatic number of a graph is linked to subgraphs with high total Betti number. One beautiful trial balloon towards this was a conjecture just verified by Bonamy, Charbit and Thomass\'e \cite{bct}: Graphs with large chromatic number induce $3k$--cycles. The connection is that $ {\mathbf{b}}(C_n)$ is 2 if $n$ is divisible by three, and 1 otherwise.

In this paper we continue the study the path of understanding the connection between ${\mathbf{b}}(G)$ and the cycle and colouring structure of $G.$ An explicit corollary of our main result is that if $G$ is a graph with at most $k \geq 2$ vertex disjoint cycles, then
\[
{\mathbf{b}}(G) \leq \left( \frac{1+\sqrt{5}}{2} \right) ^ {(2+o(1)) k \log k}.
\]

Ramanujan graphs of order $n$ and chromatic number at least $\chi$ have a girth of at least order $\frac{\log n}{3 \log \chi}.$ Any subgraphs of that order would have ${\mathbf{b}}(G) \leq 1$ and be useless to detect a high chromatic number. We show that the subgraphs 
with less than
\[ \frac{ \log n }{3  \log \chi } n^{0.003( \log \chi)^{-1}} \]
vertices also have too low ${\mathbf{b}}(G),$ because there is a planar graph of the same order with higher total Betti number.

Our results on the total Betti number all builds on creating discrete Morse functions in a combinatorial way, and using that 
${\mathbf{b}}(G) \leq {\mathbf{c}}(G)$ where $ {\mathbf{c}}(G)$ is the minimal number of critical cells in an acyclic matching on the face poset of the independence complex of the graph $G.$ In Section 2 we introduce and prove basic results on how to compute inequalities for ${\mathbf{c}}(G)$ by splitting $G$ into other graphs. In Section 3 we then apply that together with techniques from Voss' proof \cite{V69} of the Erd\"os-P\'osa theorem to prove our main results, which is an exact version of the corollary stated above. This also involves some cute enumerative combinatorics on Lucas numbers. In Section 4 we recall Ramanujan graphs and insert its parameters in our general result and compare it to results on planar lattices.

\section{Discrete Morse theory}

\subsection{Basic discrete Morse theory}
In this section we review some elementary results about discrete Morse theory and introduce one piece of new notation in Definition~\ref{def:c}. A \emph{simplicial complex} on a set $V$, which in this paper always is finite, is a set $\Delta$ of subsets of $V$ that is closed under taking subsets. In particular, if $\Delta$ is not the empty set, then it contains the empty set. We regard simplicial complexes as they are, and do not follow the practice of some topologists to remove the empty set. The boolean lattice on $V$ induces a poset structure on $\Delta$ by inclusion, and we also denote the \emph{(face) poset} by $\Delta.$ An element \emph{covers} another one in a poset if it is strictly larger, but minimal with that property. For a set $S$ let ${S \choose k}$ be the set of subsets of $S$ of cardinality $k.$ A \emph{matching} $M$ on $\Delta$ is a subset of ${\Delta \choose 2}$ satisfying that
\begin{itemize}
\item[--] if $\{ \sigma, \tau \} \in M$ then $\tau$ covers $\sigma$ in the poset $\Delta$ (or the other way around)
\item[--] and if $m_1,m_2\in M$ and $m_1 \cap m_2 \neq \emptyset$ then $m_1=m_2.$
\end{itemize}
A sequence of distinct elements $\{\sigma_1 \subset \tau_1\}, \ldots, \{\sigma_n \subset \tau_n\}$ of a matching $M$ on $\Delta$ with $n>1$ is a \emph{cycle} if $\tau_i \supset \sigma_{i+1}$ for each $1\leq i < n$ and $\tau_n \supset \sigma_1.$ A matching is \emph{acyclic} (also called a Morse matching) if there are no cycles in it. 
The elements not contained in a matching are called \emph{critical}. If there are no critical elements the matching is \emph{complete}. This follows from the fundamental theorem of Discrete Morse theory.
\begin{corollary}\label{cor:dmt}
Let $\Delta$ be a simplicial complex with an acyclic matching giving $c$ critical elements. Then $\sum_i \dim \tilde{H}_i(\Delta) \leq c.$ Furthermore, if all critical elements have the same order, then it is an equality.
\end{corollary}
There are several efficient ways to construct acyclic matching and the following conceptual and elementary one was introduced in Jonsson's PhD thesis \cite{Jphd} and attributed to Bj\"orner.
\begin{lemma}\label{lem:fiber}
Let $\phi: \Delta \rightarrow Q$ be a poset map and $M=\cup_{q \in Q} M_q$ a matching on $\Delta$ with each $M_q \subseteq {\phi^{-1}(q) \choose 2}.$ If no $M_q$ contains a cycle, then $M$ is acyclic.
\end{lemma}

\subsection{Basic independent sets and their complexes}

In this text a graph $G$ is a set of vertices $V_G$ and a set of edges $E_G \subseteq {V_G \choose 2}.$ A subset $I$ of vertices is \emph{independent} if $e \subseteq I$ for no edge $e \in E_G.$ Any subset of an independent set is also independent. The set of independent subsets of a graph $G$ is $\textrm{Ind}(G),$ and it is a simplicial complex on $V_G$ called the \emph{independence complex} of $G.$ The \emph{neighbourhood} of a vertex $v$ in a graph $G$ is $N_{G}(v)=\{u\in V_g \mid \{u,v\} \in E_G \}.$

\subsection{Morse matchings on independence complexes of graphs}

Now we derive some useful results to recursively build large acyclic matchings on independence complexes of graphs. None of these results are essentially new, see \cite{Edisc, E09, Efm, Jphd, J06}, but since we can provide better and very short proofs today, they are included.

\begin{definition}\label{def:c}
For a graph $G,$ let ${\mathbf{c}}(G)$ be the minimal number of critical cells in an acyclic matching on the independence complex of $G.$ 
\end{definition}

\begin{lemma}\label{lem:empty}
If $G$ is the graph without vertices, then ${\mathbf{c}}(G)=1.$
\end{lemma}
\begin{proof}
There only element of $\textrm{Ind}(G)$ is $\emptyset$ and every pair in a matching reduces the number of critical elements by two.
\end{proof}

\begin{lemma}\label{lem:isoV}
If $G$ has an isolated vertex, then ${\mathbf{c}}(G)=0.$
\end{lemma}
\begin{proof}
Let $v$ be the isolated vertex. Consider the poset map $\phi:\textrm{Ind}(G) \rightarrow \textrm{Ind}(G\setminus v)$ by $I \mapsto I\setminus \{v\}.$ Then $\phi^{-1}(I)=\{I,I\cup \{v\}\}$ and we set $M_I=\{\{I,I\cup \{v\}\}\}$ to get a complete acyclic matching on $\textrm{Ind}(G)$ by Lemma~\ref{lem:fiber}.
\end{proof}

\begin{lemma}\label{lem:link}
If $G$ is a graph with a vertex $u,$ then ${\mathbf{c}}(G)\leq {\mathbf{c}}(G \setminus u) + {\mathbf{c}}(G \setminus (\{u\} \cup N(u))).$
\end{lemma}
\begin{proof}
Let $Q=\{\emptyset \subset \{u\}\}.$ There is a poset map $\phi: \textrm{Ind}(G) \rightarrow Q$ by $I \mapsto I \cap \{u\}.$ The fiber $\phi^{-1}(\emptyset)$ is $\textrm{Ind}(G \setminus u).$ The fiber $\phi^{-1}(\{u\})$ is isomorphic to $\textrm{Ind}(G \setminus (\{u\} \cup N(u)))$ by $I \mapsto I \setminus \{u\}.$ We compose the matchings by Lemma~\ref{lem:fiber}.
\end{proof}

\begin{lemma}\label{lem:isoE}
Let $G$ be a graph with an isolated edge $e$ and $H$ the induced subgraph of $G$ without $e.$
Then ${\mathbf{c}}(G)\leq {\mathbf{c}}(H).$
\end{lemma}
\begin{proof}
Let the edge be $uv.$ Employ Lemma~\ref{lem:link} and note that ${\mathbf{c}}(G\setminus u)=0$ by Lemma~\ref{lem:isoV} since $v$ is an isolated vertex and that $H=G \setminus (\{u\} \cup N(u)).$
\end{proof}

\begin{lemma}[The fold lemma]\label{lem:fold}
If $u\neq v$ are vertices of $G$ and $N(v) \subseteq N(u)$ then ${\mathbf{c}}(G)\leq {\mathbf{c}}(G \setminus u).$
\end{lemma}
\begin{proof}
Note that $G \setminus (\{u\} \cup N(u))$ has the isolated vertex $v$ and use Lemma~\ref{lem:isoV} and Lemma~\ref{lem:link}.
\end{proof}

\begin{proposition}\label{lem:forest}
If $G$ is a forest then ${\mathbf{c}}(G)\leq 1.$
\end{proposition}
\begin{proof}
We do induction on the number of edges of $G.$ 
First the base case of no edges in $G.$ Then it follows from Lemma~\ref{lem:empty} and Lemma~\ref{lem:isoV}.

Now the case of edges in $G.$ If there is an isolated edge then it follows from Lemma~\ref{lem:isoE} and induction. If there are no isolated edges, then there is a vertex $v$ that is a leaf of the forest $G$ incident to an edge $vw.$ The vertex $w$ is incident to another edge $wu$ because otherwise $vw$ would be an isolated edge. Now $N(v) \subseteq N(u)$ and by Lemma~\ref{lem:fold} the vertex $u$ can be folded away, and we are done by induction.
\end{proof}

This is a slightly weaker version than the main theorem of \cite{E09}.
\begin{proposition}\label{prop:feedback}
Let $G$ be a graph with a subset $U$ of vertices whose deletion removes all cycles. Then 
${\mathbf{c}}(G)\leq 2^{\# U}.$
\end{proposition}
\begin{proof}
Employ Lemma~\ref{lem:link} for every vertex in $U$ to get
$ \mathbf{c}(G) \leq \sum_{S \subseteq U} \mathbf{c}
( G \setminus  ( U \cup_{u\in S} N(u) ) )
\leq \sum_{S \subseteq U} 1 = 2^{\# U}, $
where all the $ \mathbf{c}() \leq 1$ are derived from Proposition~\ref{lem:forest}.
\end{proof}

The following useful lemma does in a sense extend the previous ones about isolated vertices and edges
that had straightforward proofs.

\begin{lemma}\label{lem:disjoint}
Let $G$ and $H$ be graphs and $G \sqcup H$ their disjoint union. Then 
${\mathbf{c}}(G \sqcup H) \leq {\mathbf{c}}(G) {\mathbf{c}}(H).$
\end{lemma}
\begin{proof}

Let $M$ be a matching on $\textrm{Ind}(G)$ with ${\mathbf{c}}(G)$ critical cells and $N$ a matching on $\textrm{Ind}(H)$ with ${\mathbf{c}}(H)$ critical cells.

For a simplicial complex $\Delta$ with an acyclic matching $L$ we construct a new poset $\Delta/L$ by identifying all elements matched by $L.$ This becomes a poset when $L$ is acyclic without further identifications exactly when $L$ is acyclic. Note that there is a natural poset map $\Delta \rightarrow \Delta/L$ whose fibers are critical elements and matchings.

As sets (and  posets) $\textrm{Ind}(G \cup H)=\textrm{Ind}(G) \times \textrm{Ind}(H).$ Consider the poset map $\phi : \textrm{Ind}(G \cup H) \rightarrow \textrm{Ind}(H)/N$ given by composing the forgetful map $\textrm{Ind}(G \cup H) \rightarrow \textrm{Ind}(H)$ and the quotient map $\textrm{Ind}(H) \rightarrow \textrm{Ind}(H)/N.$ For any matching $m=\{J \subset J \cup\{v\}\}$ the fiber $\phi^{-1}(m)$ is isomorphic to $\textrm{Ind}(\{v\} \cup H)$ by $I \mapsto I\setminus J.$ By Lemma~\ref{lem:isoV} there is a complete matching on $\textrm{Ind}(\{v\} \cup H)$ since $v$ is an isolated vertex. For any critical element $J$ of $\textrm{Ind}(H)$ under the matching $N$ the  fiber $\phi^{-1}(J)$ is isomorphic to $\textrm{Ind}(G)$ by the map $I \mapsto I \setminus J.$ For $\textrm{Ind}(G)$ we have a matching $M$ given.

With this construction every pair of critical elements of $\textrm{Ind}(G)$ and $\textrm{Ind}(H)$ gives exactly one critical element of $\textrm{Ind}(G \cup H)$ by taking their union, and those are all critical elements.
\end{proof}

\section{Independence complexes of graphs with a given number of vertex disjoint cycles}

We start by deriving an enumerative lemma required later on. The \emph{Fibonacci numbers} are defined by 
$\varphi(0)=0, \varphi(1)=1$ and $\varphi(n)=\varphi(n-1)+\varphi(n-2)$ for any integers $n,$ and the \emph{Lucas numbers} are defined by $\ell(n)=\varphi(n-1)+\varphi(n+1).$ We have defined these numbers for all integers, including the negative ones. 
For $n>1$ the Lucas number $\ell(n)$ is the integer closest to $(\frac{1 + \sqrt{5}}{2})^n.$ Note that $\varphi(n)\geq 0$ for $n\geq -1$ and the recursion can be formulated by matrices as
\[ \left(
\begin{array}{cc}
\varphi(n+1) & \varphi(n) \\
\varphi(n) & \varphi(n-1) \\
\end{array}
\right)
=
\left(
\begin{array}{cc}
1 & 1 \\
1 & 0 \\
\end{array}
\right)^n.
\]
The Lucas numbers enumerates matchings on cycles and this can be viewed as a generalisation of that.
\begin{lemma}\label{lemma:postLucas}
If $s=(s_1,s_2,\ldots,s_n)$ is a sequence of $n\geq 2$ numbers from $\{0,1,2\}$ then the cardinality of
\[
\{(x_1,\ldots,x_n)\in \{0,1\}^n
\mid
x_1+x_2\neq s_1,\ldots, x_{n-1}+x_n\neq s_{n-1}\textrm{ and }x_n+x_1\neq s_n
\}
\]
is at most $\ell(n).$
\end{lemma}

\begin{proof}
The proof is by induction on $n.$ For any sequence $(s_1,s_2)$ there is at least one element of $\{0,1\}^2$ that is not allowed, and the cardinality of the set is at most $2^2-1=3=\ell(2).$ 

Now $n>2.$ If $s_i=1$ for some $i,$ then $x_i=x_{i+1}$ (and similarly in the case $i=n$) and we have a projection down to the case $n-1.$ By induction and $\ell(n-1)\leq \ell(n)$ this case is done. Remaining is the case of all $s_i \neq 1.$ If $s_i=0$ then
\[
\left(
\begin{array}{c}
\# \{ x_1,\ldots,x_{i+1} \textrm{ valid, }x_{i+1}=1 \} \\
\#\{x_1,\ldots,x_{i+1} \textrm{ valid, }x_{i+1}=0  \} \\
\end{array}
\right)
=
\left(
\begin{array}{cc}
1 & 1 \\
1 & 0 \\
\end{array}
\right)
\left(
\begin{array}{c}
\# \{ x_1,\ldots,x_{i} \textrm{ valid, }x_{i}=1 \} \\
\# \{ x_1,\ldots,x_{i} \textrm{ valid, }x_{i}=0 \}  \\
\end{array}
\right)
\]
and if $s_i=2$ then
\[
\left(
\begin{array}{c}
\#\{ x_1,\ldots,x_{i+1} \textrm{ valid, }x_{i+1}=1 \} \\
\#\{ x_1,\ldots,x_{i+1} \textrm{ valid, }x_{i+1}=0 \} \\
\end{array}
\right)
=
\left(
\begin{array}{cc}
0 & 1 \\
1 & 1 \\
\end{array}
\right)
\left(
\begin{array}{c}
\#\{ x_1,\ldots,x_{i} \textrm{ valid, }x_{i}=1 \}  \\
\#\{ x_1,\ldots,x_{i} \textrm{ valid, }x_{i}=0 \}  \\
\end{array}
\right).
\]
This shows that the cardinality of the set in question is $\mathrm{Tr}(A_1A_2\cdots A_n)$ where
each $A_i$ is one of the $2\times 2$ matrices determined by the value of $s_i.$ It remains to show that the trace is at most $\ell(n)$ for $n> 2.$

For integers $i,j,k,l \geq 1$ elementary Fibonacci number manipulations shows that
\[
\left(
\begin{array}{cc}
1 & 1 \\
1 & 0 \\
\end{array}
\right)^{i+k}
\left(
\begin{array}{cc}
0 & 1 \\
1 & 1 \\
\end{array}
\right)^{j+l}
-
\left(
\begin{array}{cc}
1 & 1 \\
1 & 0 \\
\end{array}
\right)^{i}
\left(
\begin{array}{cc}
0 & 1 \\
1 & 1 \\
\end{array}
\right)^{j}
\left(
\begin{array}{cc}
1 & 1 \\
1 & 0 \\
\end{array}
\right)^{k}
\left(
\begin{array}{cc}
0 & 1 \\
1 & 1 \\
\end{array}
\right)^{l}
\]
equals
\[
\varphi(j)\varphi(k)
\left(
\begin{array}{cc}
1 & 1 \\
1 & 0 \\
\end{array}
\right)^{i-1}
\left(
\begin{array}{cc}
\varphi(l-2) & \varphi(l-1) \\
\varphi(l) & \varphi(l+1) \\
\end{array}
\right),
\]
which is non-negative in all entries. By repeated application of this to bundle up the matrices by type, we have 
that 
\[ 
\mathrm{Tr}(A_1A_2\cdots A_n) \leq \mathrm{Tr}
\left(
\left(
\begin{array}{cc}
1 & 1 \\
1 & 0 \\
\end{array}
\right)^m
\left(
\begin{array}{cc}
1 & 1 \\
1 & 0 \\
\end{array}
\right)^{n-m}
\right)
\]
where $m$ is the number of zeros in the given sequence $s$ and $n-m$ is the number of twos in it.
As $m$ goes from $0$ to $n$ in the right hand side of the previous inequality
one gets a row of the \emph{Lucas triangle} \cite{josef}, whose top part looks like this:
\[
\begin{array}{ccccccccccc}
&&&&& 2\\
&&&& 1 && 1\\
&&& 3 && 2 && 3 \\
&& 4 && 3 && 3 && 4 \\
& 7 && 5 && 6 && 5 && 7 \\
11 && 8 && 9 && 9 && 8 && 11 \\
\end{array}
\]
Each row of the Lucas triangle is maximised on the boundary, where it is the Lucas number $\ell(n).$
\end{proof}

In the proof of Lemma \ref{lemma:postLucas} two $2 \times 2$ matrices are related by a non-standard transposition. Except for that, the trace inequality is very similar to the first part of Dyson's proof of the Golden-Thompson inequality \cite{dyson}. Small computer calculations supports that there is nothing particular with the matrices in the proof, except that they are related by the non-standard transposition.

Now we turn to applying discrete Morse theory to independence complexes of graphs.
The \emph{girth} of a graph is the length of its shortest cycle. In graphs of minimum degree three and a bounded number of vertex disjoint cycles the girth is also bounded. We modify a constructive graph theoretic approach by Voss \cite{V69} to estimate the girth. For our approach this is more suitable than the commonly cited work by Simonovits \cite{S}.

\begin{definition}
The \emph{effective girth} of a graph $G$ is the smallest number of vertices of degree at least three on a common cycle. For integers $k \geq 1$ let $g(k)$ be the maximal effective girth of a graph of minimal degree at least two with at most $k$ vertex disjoint cycles.
\end{definition}

\begin{proposition}\label{prop:Voss}
 $g(1)=4, g(2)=6, g(3)=8; g(k)\leq (2+o(1))\log_2(k)$ and $g(k)\leq 2 + 2 \log_2 \left( 1 + \sum_{i=1}^k g(i) \right)$
 for $k>1.$
\end{proposition}
\begin{proof}
This was proved by Voss \cite{V69} where it is stated for the girth of graphs of minimal degree three.
\end{proof}

\begin{theorem}\label{thm:removeCycle}
Let $C$ be an induced cycle in a graph $G$ with at most $n \geq 2$ vertices of $C$ adjacent to vertices of $G\setminus C$ and assume that ${\mathbf{c}}(H)\leq c$ for every induced subgraph $H$ of $G\setminus C.$ Then 
${\mathbf{c}}(G)\leq c\ell(n).$
\end{theorem}
\begin{proof}
Let $N$ be the set of $n$ vertices on $C$ that are allowed to be adjacent to vertices of $G\setminus C.$
Enumerate them $v_1,v_2,\ldots,v_n$ clockwise around $C$ and let $t_i$ be then number of vertices between
$v_i$ and $v_{i+1}$ going clockwise from $v_i$ to $v_{i+1}.$ Also let $t_n$ be the number of vertices going clockwise from $v_n$ to $v_1.$ Consider the poset map $\phi:  \textrm{Ind}(G) \rightarrow  \textrm{Ind}(G[N]).$
We want to construct an acylic matching on each fiber and then glue them together. 
The elements $\sigma \in \textrm{Ind}(G[N])$ can be encoded as sequences $(y_1,\ldots,y_n)\in \{0,1\}^n$ by $y_i=1$ if and only if $v_i\in\sigma.$ Then the fiber of $\sigma$ is isomorphic to the independence complex of a disjoint union of
\begin{itemize}
\item[(1)] a path on $t_1-y_1-y_2$ vertices,
\item[(2)] a path on $t_2-y_2-y_3$ vertices,
\item[] $\vdots$
\item[(n-1)] a path on $t_{n-1}-y_{n-1}-y_n$ vertices,
\item[(n)] a path on $t_n-y_n-y_1$ vertices, and
\item[($\ast$)] the graph $H=(G\setminus C) \setminus \cup_{v_i \in \sigma} N(v_i).$
\end{itemize}
By repeated use of Lemmas~\ref{lem:isoV}, \ref{lem:isoE}, \ref{lem:fold} one can see that $\mathbf{c}(P_n)=0$
if $n\equiv 1 (3)$ and $\mathbf{c}(P_n)=1$ otherwise, where $P_n$ is a path on $n$ vertices. By assumption
$\mathbf{c}(H)\leq c.$ By Lemma~\ref{lem:disjoint} we have an acyclic matching on the fiber with at most $c$ critical cells if all of the path lengths are not equal to one modulo three. If a single one of them is equal to one modulo three, then we have an acyclic matching without critical cells. Relaxing this problem to sum up over all
$(y_1,\ldots,y_n)\in \{0,1\}^n,$ and removing the modulo three symmetries, we get the set stated in 
Lemma~\ref{lemma:postLucas} and the factor $\ell(n)$ counting fibers with at most $c$ critical cells. By Lemma~\ref{lem:fiber} there is an acyclic matching on $\textrm{Ind}(G)$ with at most $c\ell(n)$ critical cells.
\end{proof}

\begin{theorem}\label{theorem:main}
If $G$ is a graph with at most $k\geq 0$ vertex disjoint cycles, then
\[
{\mathbf{c}}(G)\leq \ell(g(1))\ell(g(2))\cdots \ell(g(k)).
\]
\end{theorem}
\begin{proof}
The proof is by induction on $k.$ If $k=0$ then $G$ is a forest and ${\mathbf{c}}(G)\leq 1.$

Now $k>0.$ By repeated use of Lemma~\ref{lem:fold} vertices can be folded away from $G$ to reach a subgraph $G'$ such that ${\mathbf{c}}(G)\leq {\mathbf{c}}(G')$ and each vertex of degree one in $G'$ is on an isolated edge.
If $G'$ is a disjoint union of isolated vertices and edges, then $ {\mathbf{c}}(G') \leq 1$ by Lemmas~\ref{lem:isoV} and \ref{lem:isoE}. Otherwise $G'$ is a disjoint union of isolated vertices, isolated edges, and a subgraph $G''$ of minimal degree at least two, and ${\mathbf{c}}(G')\leq {\mathbf{c}}(G'')$ by Lemmas~\ref{lem:isoV}, \ref{lem:isoE} and \ref{lem:disjoint}. 

In $G''$ there is an induced cycle $C$ with at most $g(k)$ vertices of $C$ adjacent to vertices of $G'' \setminus C$ according to Proposition~\ref{prop:Voss}, and 
\[ {\mathbf{c}}(G)\leq {\mathbf{c}}(G') \leq {\mathbf{c}}(G'') \leq c \ell(g(k)) \]
by Theorem~\ref{thm:removeCycle}, where ${\mathbf{c}}(H)\leq c$ should be satisfied for all induced subgraphs $H$ of $G''.$ Any such $H$ has at most $k-1$ vertex disjoint cycles and by induction $c \leq \ell(g(1))\ell(g(2))\cdots \ell(g(k-1)).$
\end{proof}

\begin{corollary}\label{cor:main}
If $G$ is a graph with at most $k \geq 2$ vertex disjoint cycles, then
\[ 
{\mathbf{b}}(G)\leq {\mathbf{c}}(G) \leq \left( \frac{1+\sqrt{5}}{2} \right) ^ {(2+o(1)) k \log_2 k}.
\]
\end{corollary}

The graph consisting of $k$ disjoint copies of $K_5$ shows that this bound cannot be improved beyond $4^k.$ We believe that the correct exponent should be $k \log k$ rather than $k.$ Essentially this should be able to achieve by combining the probabilistic lower bound of the Erd\"os-P\'osa theorem \cite{EP65}  with Euler-characteristic calculations, since the homology should be concentrated in one dimension for a random independent complex. The resulting calculations should be doable with methods from statistical physics. Unfortunately we haven't been able to rigourously perform those calculations at this point. 

The upper bound in Erd\"os-P\'osa theorem gives an upper bound on the number of vertices that has to be removed from a graph to turn it acyclic, given that it has at most $k$ vertex disjoint cycles. The main theorem of \cite{E09} by  Engstr\"om provides an upper bound on the dimension of the total cohomology of the independence complex of a graph, given the number of vertices required to be removed to turn it acyclic. Combining those two theorems gives a bound like in Corollary~\ref{cor:main} but much worse.

For a graph $G$ without cycles we know that ${\mathbf{b}}(G)\leq 1$ is optimal. The next case is easy.
\begin{proposition}
Let $G$ be a graph without two vertex disjoint cycles. Then ${\mathbf{b}}(G)\leq 4,$ and it is optimal.
\end{proposition}
\begin{proof}
As in the proof of Theorem~\ref{theorem:main} only graphs of minimal degree at least two needs to be considered. We can also assume that $G$ is connected and not acyclic. Let $G^\ast$ be the graph attained from $G$ by the iterative contractions that replaces each vertex of degree two and its two incident edges by one edge. The minimal degree of $G^\ast$ is at least three.

If there is a triangle in $G^\ast,$ then ${\mathbf{c}}(G)\leq 1\cdot \ell(3)=4$ by Theorem~\ref{thm:removeCycle}.
If there isn't a triangle in $G^\ast,$ then by a straight forward exercise one can see that $G^\ast$ is isomorphic to a complete bipartite graph $K_{3,t}$ with $t\geq 3.$ Let $u$ and $v$ be two different vertices of $G$ that after the contractions to $G^\ast$ are still left, and they are in the part with 3 vertices. Removing $u$ and $v$ from $G$ gives an acyclic graph, and by Proposition~\ref{prop:feedback} applied to them, ${\mathbf{c}}(G)\leq 2^2=4.$

The graph $K_5$ shows that it's optimal.
\end{proof}

\section{Ramanujan graphs}

Good conditions to show strong lower bounds for the chromatic number is very hard to attain in general and several attacks have been made from topological combinatorics. A 
very interesting approach have been lifted by Kalai and Meshulam: Can a lower bound for the chromatic number of a graph be derived from that one of its subgraphs have a high total Betti number? In this section we show that there are Ramanujan graphs with arbitrary high chromatic numbers whose all small subgraphs have low total Betti number. This could be interpreted as that large enough subgraphs need to be inspected using Betti numbers to determine the chromatic number. It is quite obvious that one needs to pass a barrier of $\log N$ vertices to understand the chromatic number of order $N$ graphs, since the girth can be pushed to that order for arbitrary chromatic number. We show that at least $N^\alpha$ vertices for some $\alpha>0$ is necessary to find subgraphs with high enough Betti numbers.

The following optimal expanders were constructed by Lubotzky, Phillips and Sarnak \cite{LPS}. See Nesetril \cite{Ne} for an accessible survey on their properties as listed here.

\begin{theorem}\label{theorem:ramanujan}
For primes $p$ and $q$ with Legendre symbol $(\frac{p}{q})=1$ and $q$ sufficiently large, there are \emph{Ramanujan graphs} $X^{p,q}$ with $q(q^2-1)/2$ vertices, girth at least $2 \log_p q$ and chromatic number at least $\frac{p+1}{2\sqrt{p}}.$
\end{theorem}

\begin{proposition}
For every $n$ there is a planar graph $P$ on $n$ vertices such that for every $m\leq n$ there is a subgraph $H$ of $P$ on $m$ vertices with ${\mathbf{b}}(P)\geq 2^{(m-40\sqrt{m})/36}.$
\end{proposition}
\begin{proof}
Let $P$ be a convex piece of the hexagonal dimer lattice in Figure 1 of Adamaszek's paper \cite{Adam1}, and $H$ a convex piece of that. The ${\mathbf{b}}(P)$ estimate is from Proposition 4.1 of that paper.
\end{proof}

Any subgraph on $m$ vertices of $X^{p,q}$ would have at most $\frac{m}{2 \log_p q}$ vertex disjoint cycles due to the bound on the girth. Applying Corollary~\ref{cor:main} to it and comparing with the bound in the previous proposition proves the following.

\begin{proposition}
Let $\chi$ be a positive integer. There is a Ramanujan graph $G$ of order $n$ and chromatic number at least $\chi,$ and a planar graph $P$ of order $n,$ such that for every subgraph $G'$ of $G$ with at most  
\[
\frac{ \log_2 n }{3  \log_2 \chi }
n^{0.003( \log_2 \chi)^{-1}}
\]
vertices, there is a subgraph $P'$ of $P$ of the same order as $G'$ with
\[
{\mathbf{b}}(P') > {\mathbf{b}}(G').
\]
\end{proposition}


\begin{thebibliography}{10} 

\bibitem{Adam1}
Michal Adamaszek.
Special cycles in independence complexes and superfrustration in some lattices.
\emph{Topology Appl.} {\bf 160} (2013), no. 7, 943--950. 

\bibitem{bin}
Mireille Bousquet-M\'elou, Svante Linusson and Eran Nevo.
On the independence complex of square grids.
\emph{J. Algebraic Combin.} {\bf 27} (2008), no. 4, 423--450. 

\bibitem{bct}
Marthe Bonamy, Pierre Charbit and St\'ephan Thomass\'e. 
Graphs with large chromatic number induce $3k$--cycles.
{\tt arxiv:1408.2172}, 13 pp.

\bibitem{E05}
Hendrik van Eerten. 
Extensive ground state entropy in supersymmetric lattice models. 
\emph{J. Math. Phys.} {\bf 46} (2005), no. 12, 123302, 8 pp.

\bibitem{Edisc}
Alexander Engstr\"om.
Complexes of directed trees and independence complexes.
\emph{Discrete Math.} {\bf 309} (2009), no. 10, 3299--3309. 

\bibitem{E09}
Alexander Engstr\"om
Upper bounds on the Witten index for supersymmetric lattice models by discrete Morse theory.
\emph{European J. Combin.} {\bf 30} (2009), no. 2, 429--438.

\bibitem{Efm}
Alexander Engstr\"om.
Discrete Morse functions from Fourier transforms.
\emph{Experiment. Math.} {\bf 18} (2009), no. 1, 45--53. 

\bibitem{EP65}
Paul Erd\H{o}s and Lajos P\'osa. 
On independent circuits contained in a graph. 
\emph{Canad. Journ. Math} {\bf 17} (1965), 347--352.

\bibitem{FSE05}
Paul Fendley, Kareljan Schoutens and Hendrik van Eerten.
Hard squares with negative activity. 
\emph{J. Phys. A} {\bf 38} (2) (2005), no. 2, 315--322.

\bibitem{dyson}
Peter J. Forrester, Colin J. Thompson. 
The Golden-Thompson inequality -- historical aspects and random matrix applications.
{\tt arxiv:1408.2008}, 17 pp.

\bibitem{hui}
Liza Huijse and Kareljan Schoutens.
Supersymmetry, lattice fermions, independence complexes and cohomology theory.
\emph{Adv. Theor. Math. Phys.} {\bf 14} (2010), no. 2, 643--694. 

\bibitem{Jphd}
Jakob Jonsson.
\emph{Simplicial complexes of graphs.}
Lecture Notes in Mathematics, 1928. Springer-Verlag, Berlin, 2008. 378 pp.

\bibitem{J06}
Jakob Jonsson.
Hard squares with negative activity and rhombus tilings of the plane. 
\emph{Electron.J.Combin.} {\bf 13} (2006), no.1 1, Research Paper 67, 46 pp.

\bibitem{josef}
\v{S}\'ana Josef. 
Lucas triangle. 
\emph{Fibonacci Quart.} {\bf 21} (1983), no. 3, 192--195. 

\bibitem{Kpc}
Gil Kalai. 
Personal communication, Berkeley, November 2013; the blog post \emph{When Do a Few Colors Suffice?} at {\tt gilkalai.wordpress.com/2014/12/19/when-a-few-colors-suffice/}, December 19, 2014.

\bibitem{LPS}
Alexander Lubotzky, Ralph Phillips and Peter Sarnak.
Ramanujan graphs. 
\emph{Combinatorica} {\bf 8} (1988), no. 3, 261--277. 

\bibitem{Ne}
Jaroslav Ne\v{s}et\v{r}il.
A combinatorial classic -- sparse graphs with high chromatic number.
IUUK-CE-ITI preprint 2013-572, 29 pp.

\bibitem{S}
Mikl\'os Simonovits.
A new proof and generalizations of a theorem of Erd\H{o}s and P\'osa on graphs without k+1 independent circuits. 
\emph{Acta Math. Acad. Sci. Hungar.} {\bf 18} (1967), 191--206. 

\bibitem{V69}
Heinz-J\"urgen Voss.
Eigenschaften von Graphen, die keine $k+1$ knotenfremde Kreise enthalten.
\emph{Math. Nachr.} {\bf 40} (1969), 19--25. 

\end{thebibliography}
\end{document}